\documentclass[
a4paper]{amsart}

\def\sideremark#1{\ifvmode\leavevmode\fi\vadjust{\vbox to0pt{\vss
 \hbox to 0pt{\hskip\hsize\hskip1em
 \vbox{\hsize3cm\tiny\raggedright\pretolerance10000
 \noindent #1\hfill}\hss}\vbox to8pt{\vfil}\vss}}}%
\newcommand{\so}{\mathfrak{so}}
\newcommand{\sll}{\mathfrak{sl}}
\newcommand{\Con}{\mathcal C}
\newcommand{\C}{\mathbb C}
\newcommand{\CP}{\C\mathrm{P}}
\newcommand{\R}{\mathbb R}
\newcommand{\tG}{\tilde G}
\newcommand{\tP}{\tilde P}
\newcommand{\Bih}{\operatorname{Bih}}
\newcommand{\al}{\alpha}

\newcommand{\G}{\mathcal G}
\newcommand{\om}{\omega}
\renewcommand{\th}{\theta}
\newcommand{\Om}{\Omega}
\newcommand{\pp}{\mathfrak{p}}
\newcommand{\tom}{\tilde\om}
\newcommand{\A}{\mathcal T}
\newcommand{\g}{\mathfrak{g}}
\newcommand{\h}{\mathfrak{h}}
\newcommand{\Gr}{\operatorname{Gr}}
\newcommand{\Mat}{\operatorname{Mat}}
\newcommand{\Hom}{\operatorname{Hom}}

\newcommand{\tr}{\operatorname{tr}}

\newcommand{\mmod}{\operatorname{mod}}
\newcommand{\pr}{\operatorname{pr}}
\newcommand{\tE}{\tilde E}
\newcommand{\dd}[1]{\frac{\partial}{\partial #1}}

\newtheorem{thm}{Theorem}
\newtheorem{lem}{Lemma}

\theoremstyle{definition}
\newtheorem{ex}{Example}

\theoremstyle{remark}
\newtheorem{rem}{Remark}

\begin{document}
\title{Inclusions between parabolic geometries}
\author{Boris Doubrov \and Jan Slov\'ak}

\address{Department of Applied Mathematics, Belorussian State University,
Skoriny av. 4, 220030~Minsk, Belarus}
\email{doubrov@islc.org}
\address{Department of Mathematics and Statistics, Masaryk University, Kotl\'a\v rsk\'a
2a, 611~37~Brno, Czech Republic}

\email{slovak@math.muni.cz}

\begin{abstract}
Some of the well known Fefferman like constructions of
parabolic geometries end up with a new structure on the same manifold. In
this paper, we classify all such cases with the help of the classical
Onishchik's lists \cite{onish1} and we treat the only new series of
inclusions in detail, providing the spinorial structures on the manifolds with generic
free distributions. Our technique relies on the cohomological understanding
of the canonical normal Cartan connections for parabolic geometries and the
classical computations with exterior forms. Apart of the complete
discussion of the distributions from the geometrical point of view and the
new functorial construction of the inclusion into the spinorial geometry, we
also discuss the normality problem of the resulting spinorial connections.
In particular, there is a non--trivial subclass of distributions
providing normal spinorial connections directly by the
construction.
\end{abstract}

\thanks{\textit{Dedicated to Professor Joseph J.
Kohn on the occasion of his 75-th birthday}.\\
Research supported by the Eduard \v Cech Center for Algebra and
Geometry and the grant GACR 201/08/0397.}

\maketitle

\section{Introduction}

This paper is motivated by two recent examples of conformal structures
naturally associated with non-degenerate rank 2 vector distributions
on 5-dimensional manifolds~\cite{nurowski} and with non-degenerate
rank 3 distributions on 6-dimensional manifolds~\cite{bryant}. In both
cases there are natural parabolic geometries associated with these
distributions, which serve as an intermediate structure between the
distribution and the conformal geometry.

\begin{ex}[Nurowski~\cite{nurowski}]
Let $D$ be a rank 2 non-degenerate vector distribution on
5-dimensional manifold $M$. The non-degeneracy condition means that
the derived spaces $D^2 = D+[D,D]$ and $D^3=D^2+[D,D^2]$ have the
maximal possible dimensions $3$ and $5$ respectively. Elie Cartan
proved in his famous paper~\cite{cartan}, that there is a natural
$G_2$--geometry associated with any such distribution. Pavel Nurowski
noticed that this geometry can be extended into the conformal geometry
of signature $(3,2)$ using the classical embedding of the split real
form of $G_2$ into $\so(4,3)$. Thus, there exists a natural cone of
null-vectors $\Con\subset TM$ associated with any such distribution.
\end{ex}

\begin{ex}[Bryant~\cite{bryant}]
Similarly, let $D$ be a rank 3 vector distribution on a 6-dimensional
manifold $M$. We assume that $D$ is non-degenerate, that is
$D^2=D+[D,D]$ coincides with all $TM$. Robert Bryant~\cite{bryant} 
showed that there is a natural $SO(4,3)$--geometry associated with each
such distribution and then used the spinor representation
$\so(4,3)\to\so(4,4)$ in order to extend it to the conformal geometry of type
$(3,3)$ on the manifold $M$. In particular, there is also a natural
cone of null-vectors $\Con\subset TM$ associated with any such
distribution.
\end{ex}

Both examples have striking similarity. Firstly, they both start with
non-degenerate vector distributions, which generate parabolic geometries. All
such distributions were described by K.~Yamaguchi~\cite{yamaguchi} using the
theory of geometric structures associated with filtered
manifolds~\cite{tanaka}. Secondly, the associated parabolic geometries can be
embedded into another parabolic geometry, the conformal one.

This raises the natural question: \emph{What are all possible embeddings of
one parabolic geometry into another, and what geometric meanings do they
have?} We answer these questions, using the classical Onishchik's
list~\cite{onish1} of all possible inclusions between complex parabolic
homogeneous spaces. It appears that there are only three such examples (two
of them are series of geometries in appropriate dimensions). One of the
series provides the trivial inclusion of the contact projective geometry
(cf. D. Fox, \cite{fox}) into the projective geometry. The second example is
the embedding of $G_2$ geometry into the conformal $(3,2)$ geometry used by
P.~Nurowski. Finally, the third example produces the embedding of
$B_l$-geometry into $D_{l+1}$-geometry, which was used in the smallest
dimension $l=3$ by R.~Bryant.

We explore the latter series for an arbitrary $l$ in more detail. As a result, we
show that any non-degenerate $l$-dimensional vector distribution on $l(l+1)/2$
dimensional manifold $M$ induces a natural almost spinorial structure on $M$
associated
with it. An \emph{almost spinorial structure} on a manifold $M$ is given by an
isomorphism of the tangent bundle $TM$ with the vector bundle $\Lambda^2 S$ for
some vector bundle $S$ over $M$. This structure can be considered as a way of
identifying each tangent space $T_pM$, $p\in M$, with the set of skew-symmetric
matrices with $(l+1)$ rows and columns. In particular, for odd $l$ each such structure
defines a cone $\Con$ in $TM$ consisting of all degenerate skew-symmetric
matrices, which are given as zeros of the Pfaffian. For example, in the 
case of $l=3$ which was considered by R.~Bryant, 
we get the quadratic non-degenerate
cone. So, in this smallest possible dimension the almost spinorial structure
coincides with the conformal structure of signature $(3,3)$.

After providing a very brief review of the main concepts of parabolic
geometries, the inclusions are studied and classified for the homogeneous
models. This leads to the complete classification in Theorem \ref{cor1} and a
full understanding of the simple construction of the spinorial geometry for
the distributions. At the same time this
raises many natural questions about the relations between the two geometries.
Quite straightforward computations exploiting the understanding of the normal
Cartan connections lead to detailed description of the fundamental
invariants of the geometries in Theorems \ref{thm2} and \ref{thm3}. Finally
the study of the main geometric objects 
is continued and the normality questions are discussed in the rest of the
article, cf. Theorem \ref{t:kappa11} and Example \ref{ex-a}.

\section{General parabolic geometries}
\label{sec:2a}

The parabolic geometries
are curved deformations of the homogeneous spaces
$G/P$ with $G$ semisimple and $P$ parabolic. Thus, a
{\em parabolic geometry of
type $(G,P)$ on a manifold $M$} is a principal fiber bundle $\mathcal G\to M$
with structure group $P$, equipped with an absolute parallelism
$\om\in\Om^1(\mathcal G,\mathfrak g)$ which is Ad--invariant
with respect to the principal $P$--action and reproduces the fundamental
vector fields.

The form $\om$ is called the {\em Cartan connection} of type $(G,P)$ on $M$.
Originally, Cartan built more general absolute parallelisms by means of
his famous equivalence method.
Nowadays, the Cartan connections
appear in many areas of geometric analysis and
there is a rich theory introducing various
types of calculi and general structural results,
see \cite{capslovak} for a detailed treatment.

Most general features of the individual types of the parabolic geometries
are read off the algebraic properties of the so called {\em flat models}
$G\to G/P$, where $\om$ is the Maurer--Cartan form. On the other hand, at
curved manifolds, the parabolic geometry is rather given by some explicit
and simple structure visible at the manifold itself, while $\om$ is uniquely
determined by a construction based on a natural normalization. We have mentioned the
well known examples of projective, conformal and spinorial geometries above.

The crucial algebraic structures are derived from the grading
$\g=\g_{-k}\oplus \dots \oplus \g_k$ of the semisimple Lie algebra $\g$
giving rise to the parabolic subalgebra $\pp = \g_0\oplus\dots\oplus\g_k$.
At the level of the curved geometries, this yields the
$P$--invariant filtration on $T\G$ which projects also to the filtration on
$TM$. Let us also notice that the Cartan--Killing form identifies $\pp_+ =
\g_1\oplus\dots\oplus \g_k$ with
$(\g/\pp)^*$ as $P$--modules and $\g/\pp$ equals
to $\g_-=\g_{-k}\oplus\dots\oplus\g_{-1}$ as $G_0$--module, where $G_0$ is
the reductive part of the parabolic subgroup $P$ (with Lie algebra $\g_0$).

It is also well known, how to understand the structure of the geometries in
cohomological terms. The curvature form $\Om\in \Om^2(\G,\g)$
of the Cartan connection $\om$
is given by the structure equation
$$
\Om = d\om +\frac12[\om,\om]
$$
and the absolute parallelism allows to express the curvature by the
curvature function $\kappa:\G\to \wedge^2 \pp_+\otimes\g$, $\kappa(X,Y) =
K(\om^{-1}(X),\om^{-1}(Y))$. Thus, the curvature function
has values in the cochains of the Lie algebra cohomology of $\g_-$ with
coefficients in $\g$. This cohomology is explicitly computable by the
Kostant's version of the BBW theorem, cf. \cite{kost, silhan, capslovak}, 
and we may
compute it either by means of the standard differential $\partial$ or by its
adjoint co--differential $\partial^*$. The formula in the special case of
the above two--chains is
$$
\partial^*(Z_0\wedge Z_1\otimes X) = -Z_0\otimes [Z_1,X] + Z_1\otimes[Z_0,X]
-[Z_0,Z_1]\otimes X
.$$
Another important property of the parabolic geometries imposes conditions
on the behavior of the filtrations and is called {\em
regularity}. In words, the filtrations have to respect the Lie brackets of
vector fields. In terms of the curvature, this says that no curvature
components of non--positive homogeneities are allowed.

The Tanaka theory, further extended and worked out in last thirty years,
shows that normalizing the regular Cartan connections properly defines
an equivalence of categories of certain filtered manifolds (with
additional simple
geometric structures under some cohomological conditions, like for all
$|1|$--gradings or contact gradings) and categories of Cartan connections, cf.
\cite{tanaka, capsch, capslovak}.
Then the harmonic part of the curvature defines all the rest and, in
particular, the geometry is locally isomorphic to its flat model if and only
if the curvature vanishes. Moreover, the entire curvature tensor is
computable explicitly by a natural 
differential operator from its harmonic part, cf.
\cite{cap-twistor}.

\section{Inclusions between parabolic geometries}
\label{sec:2}

Quite often, there are natural constructions linking together
different parabolic geometries. For example,
the Fefferman's celebrated construction
of a conformal structure on a circle bundle over each hypersurface type
CR--manifold allows to exploit the much simpler invariant theory of the
conformal Riemannian structures in order to understand that of the
CR--geometry.

Of course, each such construction is of functorial character and it is
determined at the algebraic level
already at the homogeneous models. Thus we shall start at the level of Lie
groups and we describe the main ingredients of the Fefferman like
constructions.

Let $G/P$ and $\tG/\tP$ be two (real or complex analytic) parabolic
homogeneous spaces, i.e. $G$ is any semisimple Lie group and
$P\subset G$ a parabolic subgroup, and consider an
homomorphism $i\colon G\to \tG$ which is infinitesimally injective.

Second, we require that the $G$--orbit of $o=e\tilde P\in\tilde G/\tilde
P$ is open. This means that the map $\frak g\to \tilde{\mathfrak
g}/\tilde{\mathfrak
p}$ induced by $i':\frak
g\to\tilde{\frak g}$ is surjective.

Now, the subgroup $Q:=i^{-1}(\tilde P)$
is a closed subgroup of $G$, which is usually not parabolic. The
homomorphism $i$ then induces a smooth map $G/Q\to\tilde G/\tilde P$,
whose image is the $G$--orbit of $o$.

Finally, we need that $P\subset G$ contains $Q$.
Having secured all this,
there is the natural projection $\pi:G/Q\to G/P$. The
homomorphism $i:G\to\tilde G$ induces the smooth map $G/Q\to \tilde
G/\tilde P$ which is a covering of the $G$--orbit of
$o$, and as an open subset in $\tilde G/\tilde P$ carries
a canonical geometry of type $(\tilde G,\tilde P)$. This can be
pulled back to obtain such a geometry on $G/Q$.

Of course, if we replace Lie groups $G$, $\tG$ and their Maurer--Cartan
forms by the principal fiber bundles and Cartan connections, the same
construction applies with the Lie subgroups $P$, $\tP$ and $Q$, see
\cite[Section 4.5]{capslovak} for general theory and several examples.

Especially, it may happen that
$$
i(G)\tP=\tG\ \mbox{and}\
i(P)=i(G)\cap \tP
$$
i.e. $Q=P$ is the parabolic subgroup.
Then both parabolic geometries turn out to live over the same base manifold
$G/P = \tilde G/\tilde P$.
We say that $i$ is an
{\em inclusion of parabolic homogeneous spaces}. For curved
geometries of these types we talk about {\em inclusions of parabolic
geometries}.

This is equivalent to the conditions that $i(G)$
acts transitively on the manifold $\tG/\tP$ and the stationary
subgroup of this action at $o=e\tP$ coincides with $i(P)$.

All such non-trivial inclusions in complex analytic case were
described by A.~Onishchik~\cite{onish1} (see also~\cite[\S15]{onish2}
for more details). It appears that if $G$ is simple, there is a very
limited number of such examples.

\begin{thm}[Onishchik, \cite{onish1}] Let $M=G/P$ be a complex parabolic homogeneous space.
Let $\tG=(\Bih M)^o$ be the connected component of the group of all
biholomorphic automorphisms of $M$. If $G$ is simple, then $\tG$ is also
simple. Moreover, $\tG$ always coincides with $G$ with the following
exceptions:
\begin{enumerate}
\item $G=PSp(2l,\C)$, $P=P_\Sigma$, where $\Sigma=\{\al_2,\dots,\al_l\}$, $M=\CP^{2l-1}$, $\tG=PSL(2l,\C)$;
\item $G=G_2$, $P=P_{\{\al_2\}}$, $M=Q^5$, $\tG=PSO(7,\C)$;
\item $G=SO(2l+1,\C)$, $P=P_\Sigma$, where $\Sigma=\{\al_1,\dots,\al_{l-1}\}$, $M=I^o\Gr_{l+1}(\C^{2l+2})$, $\tG=PSO(2l+2,\C)$.
\end{enumerate}
\end{thm}

The first exceptional geometries are the complex versions of the so called
projective contact structures.
The real split form of the symplectic algebra
is the only one allowing this complexified parabolic subalgebra, cf.
\cite[Section 2.3]{capslovak}. The inclusion can be also nicely interpreted with the help
of the distinguished geodesics related to the projective geometry and these
questions have been studied in great detail by D. Fox, \cite{fox}.

In  the second case, $Q^5$ denotes the quadric
in $\CP^6$ given by the equation $(z,z)=0$,
where the scalar product on $\C^7$ is given by the standard $SO(7,\C)$
representation, and $G_2$ is embedded into $SO(7,\C)$ by its unique (up to the
conjugation) irreducible 7-dimensional representation, which has an invariant
non-degenerate symmetric form.
Thus this corresponds to the Nurowski's
example above. Again, the split real form is the only one allowing this
complexified parabolic subalgebra.

The space $I^o\Gr_{l+1}(\C^{2l+2})$ denotes the connected component of
the manifold of isotropic Lagrangian subspaces in $\C^{2l+2}$, which
contains $V_0=\langle e_1,\dots,e_{l+1}\rangle$. Here
$\{e_1,\dots,e_{2l+2}\}$ is a basis in $\C^{2l+2}$, such that
$SO(2l+2,\C)$-invariant symmetric form has the matrix:
\[
\begin{pmatrix} 0 & E_{l+1} \\ E_{l+1} & 0 \end{pmatrix}.
\]
In order to understand better
the third example in the list of exceptions, let us work out
the explicit description of the algebraic inclusion. Again, only the split
real form allows for such parabolics and so we shall deal with these real
Lie algebras.


Let us identify $SO(l,l+1)$ with the set
of matrices preserving the following symmetric form:
\[
\begin{pmatrix}
0 & 0 & E_{l} \\
0 & 1 & 0 \\
E_{l} & 0 & 0
\end{pmatrix}
.
\]
Then the embedding $i\colon G\to\tG$ can be described infinitesimally by the
following injective mapping of the corresponding Lie algebras:
\begin{equation}\label{incl}
\begin{gathered}
\alpha\colon \so(l,l+1)\mapsto\so(l+1,l+1),
\\
\begin{pmatrix}
A & X & Y \\
-{}Z^t & 0 & -{}X^t \\
T & Z & -{}A^t
\end{pmatrix}
\mapsto
\begin{pmatrix}
A & \frac{1}{\sqrt{2}} X & \frac{1}{\sqrt{2}} X & Y \\
- \frac{1}{\sqrt{2}} Z^t & 0 & 0 & - \frac{1}{\sqrt{2}} X^t \\
- \frac{1}{\sqrt{2}} Z^t & 0 & 0 & - \frac{1}{\sqrt{2}} X^t \\
T & \frac{1}{\sqrt{2}} Z & \frac{1}{\sqrt{2}} Z & -A^t
\end{pmatrix}
\end{gathered}
\end{equation}
where $A, Y, T\in\Mat_{l}(\R)$, $X,Z\in \R^l$, $Y+Y^t=T+T^t=0$.

Clearly, the $P$--module structure on $\g$ reveals that the filtration of
a parabolic geometry of type $(G,P)$ is given by the distribution of
rank $l$ on a manifold of dimension $\frac12(l+1)l$. Since the first
cohomology $H^1(\g_-,\g)$ concentrates in negative homogeneities only, the
filtration determines the normal parabolic geometry completely, cf.
\cite[Section 4.3]{capslovak}. Concerning the geometry of type $(\tG,\tP)$,
we have mentioned already that this is one of the examples of geometries
with trivial filtration and determined by a classical G--structure, which is
given by an identification of $TM$ with the second exterior tensor power
of an auxiliary vector bundle $S$ of dimension $l+1$, cf.
\cite[Section 4.1]{capslovak}.

\begin{thm}\label{cor1}
The only inclusions of real
parabolic geometries with simple Lie groups $G$ and
$\tG$ are the following ones:
\begin{enumerate}
\item The obvious projective structure induced by the contact projective
geometries (see \cite{fox} and \cite[Section 4.5]{capslovak}).
\item The Cartan's example of distributions with grows vectors $(2,3,5)$ on
five--dimensional manifolds carrying the conformal Riemannian geometry of
signature $(3,2)$ (see \cite{nurowski}).
\item The generic free distribution with grows vector $(l,\frac12(l+1)l)$,
$l\ge 3$ carrying the spinorial structures (well known only in
dimension $l=3$, \cite{bryant}).
\end{enumerate}
\end{thm}

\begin{proof} If an inclusion of parabolic geometries of given types should
exist, then there must be the corresponding inclusion of the homogeneous
spaces. However, at the level of the Lie groups everything is realized by
real analytic objects. Thus, complexifying, there must exist the appropriate
inclusion in the holomorphic category and the Onishchik's list together with
the above observations complete the proof.
\end{proof}

In the rest of this paper, we shall work out more details on the new series of
examples. In order to understand the functorial construction of the spinorial
geometry from the rank $l$ distribution, we need a bit more knowledge of both
geometries. On the other hand, we shall see that the $G_0$--module structure of $\g_-$ and the
standard representations of $G$ and $\tG$ are enough to construct quickly the
almost spinorial structure on $M$ directly from the Cartan connection
$(\G,\om)$ associated with the rank $l$ distribution $D$ on $M$. The more
interesting and difficult questions are:
\begin{itemize}
\item \emph{How much of
the Cartan connection $\om$ do we need to recover the spinorial geometry?}
\item \emph{Under which conditions will the induced spinorial Cartan connection
$\tilde \om$ be normal again?}
\end{itemize}
We shall come back to these questions in the subsequent sections.

Now, let us consider the
associated \emph{standard tractor bundle} $\A M = \G \times_P V$, where $V$ is
the standard $SO(l, l+1)$ representation. The action of the parabolic subgroup
$P$ preserves the filtration $V=V^0\supset V^1\supset V^2\supset {0}$, where
$\dim V^2 = l$, $\dim V^1 = l+1$. It induces the filtration of the tractor
bundle:
\[
\A M =  \A^0M\supset \A^1M\supset \A^2M \supset 0.
\]

As a $G_0$ module, the standard representation splits as $V = \R^l \oplus \R
\oplus \R^l$, where $\R$ is the trivial representation while $\R^l = V/V^1$
is isomorphic to $\g_{-1}$. Thus, the $G_0$--module
$$
S=V/V^2 = \R\oplus \R^l
$$
has the property $\wedge^2 S=\g_-$. It is easy to see, that this $G_0$
module structure is compatible with the inclusion $G_0\to \tilde G_0$.

At the level of the parabolic geometry determined by the distribution $D$,
we simply consider the auxiliary vector bundle $\mathcal S = \G\times_P S$ and
we see that the tangent bundle $TM$ is naturally isomorphic to $\wedge^2
\mathcal S$. Since the identification is compatible with the inclusion of
the reductive parts of the parabolic subgroups, this is the right spinorial
structure as obtained from the general construction.

\section{Canonical Cartan connection for length 2 distributions}
\label{sec:3}

Let us notice, that we have not exploited the entire Cartan connection $\om$
in the construction above. Rather we have only used the splitting of the
$G_0$--modules $\g_-$ and $V$. Moreover, only the homogeneity one part of
the total splitting of $V$ was necessary (we have split $V/V^2$ only).

In the Cartan--Tanaka procedure, this amount of information is obtained
after the first prolongation step (the bottom up approach). The construction
in \cite[Section 3.1]{capslovak} provides the entire Cartan connection and
the complete information on the structure of the curvature, without the
explicit prolongation steps. We shall combine
these two approaches by using the detailed knowledge on the curvature during
the explicit prolongation computations. For the sake of simplicity, we shall
ignore the lowest dimensional case with $l=3$ since the curvature structure
is different and this case is well known.

Let $D$ be a rank $l\ge4$ vector distribution on a manifold $M$ of dimension
$l(l+1)/2$. We say that $D$ is non-degenerate, if $D+[D,D]=TM$. In other words, if
$X_1, \dots, X_l$ is any local basis of sections of $D$, then the vector fields
$X_i, [X_j,X_k]$, $1\le i \le l$, $1\le j<k\le l$ should form a basis of the
tangent space $TM$ at all points where sections $X_i$ are defined.

Under this non-degeneracy condition, the general theory implies that there
is a natural regular and normal Cartan connection
of type $(G,P)$ on $M$, with the corresponding pair of Lie algebras $(\g,\pp)$
given by:
$$
\g = \left\{\begin{pmatrix}
A & X & Y \\
-{}Z^t & 0 & -{}X^t \\
T & Z & -{}A^t
\end{pmatrix} \right\},\qquad
\pp = \left\{\begin{pmatrix}
A & 0 & 0 \\
-{}Z^t & 0 & 0 \\
T & Z & -{}A^t
\end{pmatrix}\right\},
$$
where $A,\ Y,\ T\in\Mat_{l}(\R),\ X,\ Z\in \R^l,\
Y+Y^t=T+T^t=0$.

\begin{thm}[\cite{tanaka}, \cite{capslovak}]\label{thm2}
For each non--degenerate distribution of rank $l$ on a manifold of dimension
$\frac12(l+1)l$, there
is the unique regular normal Cartan connection
of type $(G,P)$ on $M$ (up to isomorphisms).
The only fundamental invariant of these
parabolic geometries is concentrated in the homogeneity degree $1$ and
corresponds to the totally trace-free part of the $\sll(l,\R)$-submodule
$\Hom(\g_{-1}\wedge\g_{-2},\g_{-2})$ in the curvature.
\end{thm}

\begin{proof}
The existence and uniqueness 
of the Cartan connection follow from the Tanaka theory (see also
\cite[Section 3.1]{capslovak}), for
the explicit computation of the curvature see \cite[Section 4.3]{capslovak}
or compute following the algorithm of Kostant, see \cite{kost, silhan}.
\end{proof}

In more detail, let $\pi_1, \dots,
\pi_{l-1}$ be the fundamental weights of $\sll(l,\R)$. Then $\g_{-1}$ has
weight $\pi_1$, $\g_{-2}$ is isomorphic to $\wedge^2\g_{-1}$ and has weight
$\pi_2$, and by the trace-free part of $\Hom(\g_{-1}\wedge\g_{-2},\g_{-2})$ we
mean the unique submodule with the highest weight $\pi_2 + \pi_{l-2} +
\pi_{l-1}$.

\begin{rem}
Unlike the case of rank 3 distributions, the fundamental invariant for this
Cartan connection is a part of the torsion for distributions of rank $l\ge4$.
In particular, in this case any regular and
normal torsion-free geometry is automatically
flat, i.e. locally isomorphic to its homogeneous model.
\end{rem}

Now we start our computations.
Let $\{X_1,\dots,X_l\}$ be any (local) frame of $D$. Denote then by $X_{[ij]}$
vector fields $-[X_i,X_j]$. Then vector fields $\{ X_i, X_{[jk]} \}$ will form
the frame on $M$. Denote by $\{ \th^i, \th^{[jk]} \}$ the dual coframe and by
$D^{\perp}$ the set of all 1-forms on $M$ annihilating $D$. It is clear that
$D^{\perp}$ is generated by $\th^{[jk]}$.

Note that
\[
d\th^{[jk]}(X_j, X_k) = - \th^{[jk]}([X_j,X_k]) = 1.
\]
This implies that
\[
d\th^{[jk]} = \th^j\wedge\th^k \mod \langle \th^{[rs]}\rangle.
\]
So, the structure equations of the coframe $\{\th^i,\th^{[jk]}\}$ have
the form:
\begin{equation}\label{streq}
\begin{aligned}
d\th^r &= f^r_{i[jk]}\th^i\wedge\th^{[jk]} + f^r_{[[ij][kl]]}\th^{[ij]}\wedge \th^{[kl]}, \\
d\th^{[rs]} &= \th^r\wedge \th^s + f^{[rs]}_{i[jk]}\th^i\wedge\th^{[jk]} +
f^{[rs]}_{[[ij][kl]]}\th^{[ij]}\wedge \th^{[kl]},
\end{aligned}
\end{equation}
where $f^r_{i[jk]}$, $f^r_{[[ij][kl]]}$, $f^{[rs]}_{i[jk]}$,
$f^{[rs]}_{[[ij][kl]]}$ are the structure functions of the coframe $\{\th^i,
\th^{[jk]}\}$ on $M$ uniquely determined by the choice of the frame
$\{X_1,\dots,X_l\}$. Note that these families of functions do not form any
tensor, since their transformation rule under the change of the frame involves
derivatives. The natural Cartan connection associated with the distribution $D$
will allow us to construct the coframes behaving much nicer and we shall
obtain the components of the curvature tensor at the same time.

Let $\pi\colon \G\to M$ be any principle $P$-bundle on $M$ and
$\om\colon T\G\to \g$ any regular Cartan connection of type $G/P$. For any
section $s\colon M\to\G$ we can write explicitly:
\[
s^*\om = \begin{pmatrix} \om^i_j & \om^i & \om^{[ij]} \\ -\om_j & 0 &
  -\om^i \\ \om_{[ij]} & \om_j & -\om^j_i\end{pmatrix}
\]
where $\om^{[ij]}$, $\om_{[ij]}$, $\om^i_j$, $\om^i$, $\om_j$ are 1-forms on
$M$.

We say that the Cartan connection $(\G,\om)$ is \emph{adapted to the
distribution $D$}, if $D=\langle \om^{[ij]}\rangle^\perp$. It is easy
  to see that this definition does not depend on the choice of the
  section $s$.

Let $(G,\om)$ be an adapted Cartan connection. Then we have
\[
 D^\perp = \langle \om^{[ij]} \rangle = \langle \th^{[ij]} \rangle,
\]
where, as above, the forms $\th^{[ij]}$ are defined by fixing a frame
$\{X_1,\dots,X_l\}$ on $D$.

We can always choose such section $s\colon M \to \G$ that
\begin{equation}\label{norm0}
\om^i = \th^i \mod D^{\perp}.
\end{equation}
This condition defines $s$ uniquely up to the transformations $s\to sg$, where
$g\colon M\to P_{+}$ is an arbitrary $P_{+}$-valued gauge transformation.

Consider the component $\Om^{[ij]}$ of the curvature tensor:
\[
\Om^{[ij]}=d\om^{[ij]} - \om^i\wedge\om^j+\om^i_k\wedge
\om^{[kj]}-\om^j_k\wedge\om^{[ik]}
\]
(here and below we use the Einstein summation convention).

Let us remind that $\om$ is regular and so only positive homogeneities may
appear in the curvature. This immediately implies that
\[
d\om^{[ij]} = \om^i\wedge\om^j = \th^i\wedge\th^j \mod D^{\perp},
\]
and, hence
\begin{equation}\label{norm01}
\om^{[ij]}=\th^{[ij]}\quad \mbox{for all $1\le i<j\le l$.}
\end{equation}

Compute now the curvature coefficients of degree $1$ together with the
section normalizations of the form $s\mapsto \tilde s = sg$, where $g$ takes
values in $\exp{\g_1}$. Any such transformation leads to the
following transformation of the pull-back forms $\tilde s^*\om$:
\begin{align*}
\tom^{[ij]} & = \om^{[ij]},\\
\tom^{i} & = \om^i + p_j\om^{[ij]},
\end{align*}
where functions $p_j$ define the mapping $dg\colon M\to \g_1$.
Assume that
\begin{align*}
\om^i &= \th^i + C^i_{[jk]}\om^{[jk]},\\
\om^i_j &= A^i_{kj}\om^k \mod D^{\perp}
\end{align*}
for some functions $A^i_{kj}$, $C^i_{[jk]}$. They are transformed by the gauge
transformation $g$ according to the following formula:
\begin{align*}
\tilde A^i_{jk} & = A^i_{jk} - \delta^i_jp_k;\\
\tilde C^i_{[jk]} & = C^i_{[jk]} + \delta^i_{[j} p_{k]}.
\end{align*}
We can always find such functions $p_i$ that $\sum_i A^i_{ik} = 0$. This
determines the section $s$ uniquely up to the transformations $s\mapsto sg$,
where $g$ takes values in $\exp(\g_2)$.

We have:
\begin{align*}
\Om^{[ij]} &= P^{[ij]}_{r[st]}\om^r\wedge\om^{[st]} \mod
\wedge^2D^\perp,\\
\Om^i &= Q^i_{[rs]}\om^r\wedge\om^s \mod \Lambda^1(M)\wedge D^{\perp}.
\end{align*}
where
\begin{align*}
P^{[ij]}_{r[st]} &= f^{[ij]}_{r[st]} + \delta^{[i}_{[s}A^{j]}_{rt]} +
\delta^{[i}_r C^{j]}_{[st]},\\
Q^{i}_{[jk]} &= C^i_{[jk]}+A^i_{[jk]}.
\end{align*}
Assuming that the connection does not have the torsion in the term
$\Hom(\wedge^2\g_{-1},\g_{-1}$) (as it is implied from the normality assumption
via Kostant theorem), we get $Q^i_{[jk]}=0$, that is
\[
C^i_{[jk]} = - A^i_{[jk]}.
\]

Substituting this equality to the expression of $P^{[ij]}_{r[st]}$ we get:
\begin{equation}\label{tensorP}
P^{[ij]}_{r[st]} = f^{[ij]}_{r[st]} +
\delta^{[i}_{[s}A^{j]}_{rt]}-\delta^{[i}_r A^{j]}_{[st]}.
\end{equation}

Let us remind that $l\ge4$. It turns out that we can uniquely
determine coefficients $A^i_{jk}$ by the following two conditions:
\begin{itemize}
\item $\sum_i A^i_{ik}=0$ according to the choice of the section $s$;
\item tensor $P^{[ij]}_{r[st]}$ is totally trace-free.
\end{itemize}
Indeed, the traces of $A$'s do not contribute in the right hand side and the
trace--free condition on $P$ just determines the rest.

\begin{thm}\label{thm3}
The trace-free part of the tensor $P^{[ij]}_{r[st]}$ computed above
is the only fundamental invariant of the non--degenerate rank $l$ distribution
$D$ on a manifold of dimension $\frac12(l+1)l$. Hence, the Cartan
connection associated with $D$ is flat if and only if this tensor vanishes
identically.
\end{thm}

\begin{proof}
The statement is a direct consequence of Theorem~\ref{thm2} and the
computations above.
\end{proof}

In particular, this gives us an explicit condition when an
arbitrary non-degenerate distribution $D$ of rank $l\ge4$ is locally equivalent
to the left-invariant distribution on the nilpotent Lie group corresponding
to the algebra~$\g_{-}$.

\begin{rem}
Let us also comment on the link between the
coefficients $A$ and $C$.
In general terms, these objects correspond to the choice of partial
Weyl connections (the coefficients $A^i_{jk}$) and the splittings of the
filtration (the improvement of the coframe by deforming $\th^i$), cf.
\cite[Chapter 5]{capslovak}. Both of
these objects have to be fixed together because they influence the same
curvature components in homogeneity one. This has been reflected by the
explicit link between $A$'s and $C$'s.

Having fixed these homogeneity one objects, all the necessary ingredients
for the construction of the spinorial geometry are available, see the end of
Section 3 above. Still, in general, we
cannot replace the normal Cartan connection of the distribution by 
the simpler construction of the spinorial normal Cartan connection
(although regularity is not an issue for one--graded geometry) since we do not
know whether the normality will be preserved.
\end{rem}

\section{Computation of coefficients of degree 2}
Let as continue the computation of the regular normal Cartan connection $\om$
associated with the distribution $D$ and its curvature tensor $\Om$. Let us now
consider curvature coefficients of degree~2 together with the normalization of
the section $s\mapsto sg$, where $g\in \exp(\g_2)$. Any such transformation
leads to the following transformation of the form $\tom$:
\begin{align*}
\tom^i_j &= \om^i_j + q_{[ik]}\om^{[kj]},\\
\tom_i  &= \om_i + q_{[ik]}\om^k \mod D^{\perp}.
\end{align*}
Assume that
\begin{align*}
\om^i_j &= A^i_{kj}\om^k + E^i_{j[kl]}\om^{[kl]},\\
\om_i &= F_{ki}\om^k \mod D^\perp,
\end{align*}
where the functions $A^i_{kj}$ were determined above, and $E^i_{j[kl]}$,
$F_{ki}$ are functions to be found. They are transformed by the gauge
transformation $g$ by the following formulas:
\begin{align*}
\tilde E^i_{j[kl]} &= E^i_{j[kl]} + \delta^i_kq_{[jl]}-\delta^i_lq_{[jk]},\\
\tilde F_{ki} &= F_{ki} + q_{[ki]}.
\end{align*}
We see that we can make $F_{kl}$ symmetric by an appropriate choice of
$q_{[kl]}$. This determines uniquely the section $s\colon M\to\G$.

Let us now proceed with computations of the curvature coefficients of degree 2.
We have
\begin{multline*}
\Om^{[ij]} = d\om^{[ij]} -\om^i\wedge\om^j +
\om^i_k\wedge\om^{[kj]}-\om^j_k\om^{[ik]} = \\
f^{[ij]}_{[[kl][rs]]}\om^{[kl]}\wedge\om^{[rs]} +
E^i_{k[rs]}\om^{[rs]}\wedge\om^{[kj]}-E^j_{k[rs]}\om^{[rs]}\wedge\om^{[ik]} =
R^{[ij]}_{[[kl][rs]]}\om^{[kl]}\wedge\om^{[rs]},
\end{multline*}
where
\begin{equation}\label{tensorR}
R^{[ij]}_{[[kl][rs]]} = f^{[ij]}_{[[kl][rs]]} +
\delta^{[i}_{[l}E^{j]}_{k][rs]}.
\end{equation}
Next,
\begin{multline*}
\Om^i = d\om^i + \om^i_k\wedge\om^k+\om^{[ik]}\wedge\om_k = \\
d\left(\th^i - A^i_{[jk]}\om^{[jk]}\right) +
\om^i_k\wedge\om^k+\om^{[ik]}\wedge\om_k = (\mmod \Lambda^2D^{\perp})\\
f^i_{r[st]}\om^r\wedge\om^{[st]} - \frac{\partial A^i_{[jk]}}{\partial
\om^l}\om^l\wedge\om^{[jk]} - A^i_{[jk]}f^{[jk]}_{r[st]}\om^r\wedge\om^{[st]} +
\\
E^i_{k[rs]}\om^{[rs]}\wedge\om^k+F_{jk}\om^{[ik]}\wedge\om^j =
S^i_{j[kl]}\om^j\wedge\om^{[kl]},
\end{multline*}
where
\begin{equation}\label{tensorS}
S^i_{j[kl]} = f^i_{j[kl]} - \frac{\partial A^i_{[kl]}}{\partial \om^j} -
A^i_{[rs]}f^{[rs]}_{j[kl]} - E^i_{j[kl]}-F_{j[l}^{}\delta^i_{k]}.
\end{equation}

Finally,
\begin{multline*}
\Om^i_j = d\om^i_j +
\om^i_r\wedge\om^r_j-\om^i\wedge\om_j+\om^{[ik]}\wedge\om_{[kj]} = (\mmod D^\perp)\\
d(A^i_{kj})\om^k + A^i_{kj}d(\th^k - A^k_{rs}\om^{[rs]}) +
E^i_{j[kl]}d\om^{[kl]} + \\
\left(A^i_{sr}\om^s\right)\wedge \left(A^r_{tj}\om^t\right) -
F_{kj}\om^i\wedge\om^k = T^i_{j[kl]}\om^k\wedge\om^l,
\end{multline*}
where
\begin{equation}\label{tensorT}
T^i_{j[kl]} = - \frac{\partial A^i_{[kj]}}{\partial \om^l} - A^i_{rj}A^r_{[kl]}\\
 + A^i_{[kr}A^r_{l]j}+E^i_{j[kl]}-\delta^i_{[k}F_{l]j}^{}.
\end{equation}

Let us now compute normalization conditions on coefficients $R,S,T$ that are
implied by the normality condition. In order to do this we explicitly compute
$\partial^*$ for each of these tensors.

Fix a basis in $\g$: $\g_{-2}=\langle E_{[ij]} \rangle$,
$\g_{-1}=\langle E_{i} \rangle$,
$\g_0=\langle E^i_j \rangle$,
$\g_1=\langle E^i \rangle$,
$\g_2=\langle E^{[ij]} \rangle$.
Let us note that Killing form $B$ on $\g$ is given by $\tr XY$, $X,Y\in\g$,
and, in particular, we have:
\begin{align*}
B(E_{[ij]}, E^{[ij]}) &= -2;\\
B(E_i, E^i) &= -2;\\
B(E^i_j, E^j_i) &= 2.
\end{align*}
For simplicity we can always multiply $B$ by $-1/2$ and assume that the dual
elements to $E_i$ and $E_{[jk]}$ are equal to $E^i$ and $E^{[jk]}$
respectively.

Using the identification $\g_{-}^*$ with $\pp_{+}$, we can write the tensor
$R\in \Hom(\wedge^2\g_{-2},\g_{-2})$ as an element of
$\wedge^2\g_2\otimes\g_{-2}$:
\[
 R = R^{[ij]}_{[[kl][rs]]} E^{[kl]}\wedge E^{[rs]}\otimes E_{[ij]}.
\]

Then we have:
\begin{multline*}
\partial^* R = R^{[ij]}_{[[kl][rs]]} \left(E^{[kl]}\otimes [E^{[rs]}, E_{[ij]}] - E^{[rs]}\otimes [E^{[kl]},
E_{[ij]}]\right) =\\
 \sum_j R^{[ij]}_{[[kj][rs]]}E^{[rs]}\otimes E^k_i.
\end{multline*}
So, we see that $\partial^* R$ lies in $\g_2\otimes\g_0$ and can be
symbolically written as $\partial^* R = \tr R$.

Similarly, we have
\[
 S = S^i_{j[kl]} E^j\wedge E^{[kl]}\otimes E_i,
\]
and
\begin{multline*}
\partial^*S = S^i_{j[kl]} \left(E^j\otimes [E^{[kl]}, E_i] - E^{[kl]}\otimes [E^j,
E_i]\right)= \\
\sum_i S^i_{j[il]}E^j\otimes E^l - S^i_{j[kl]}E^{[kl]}\otimes E^j_i.
\end{multline*}
Thus, we see that $\partial^*S$ lies in $\g_1\otimes\g_1 + \g_2\otimes\g_0$ and
can be symbolically written as $\partial^* S = \tr S - S$.

Finally, we have
\[
T = T^i_{j[kl]} E^k\wedge E^l\otimes E^j_i,
\]
and
\begin{multline*}
\partial^*T = T^i_{j[kl]}\left(E^k\otimes [E^l, E^j_i\right) - E^l\otimes
[E^k,E^j_i] - E^{[kl]}\otimes E^j_i)= \\
\sum_i T^i_{j[il]} E^j\otimes E^l - T^i_{j[kl]}E^{[kl]}\otimes E^j_i.
\end{multline*}
Consequently, $\partial^*T$ lies in $\g_1\otimes\g_1 + \g_2\otimes\g_0$ and
can be symbolically written as $\partial^*T = \tr T - T$.

Summarizing, we get
\begin{lem} Let $\kappa_2$ be the degree 2 part of the curvature of the
regular and normal $B_l$ geometry
associated with a non-degenerate distribution $D$.
Then it can be expressed as
a sum of three tensors:
\begin{align*}
R &\in \Hom(\g_{-2}\wedge\g_{-2},\g_{-2}),\\
S &\in \Hom(\g_{-1}\otimes\g_{-2},\g_{-1}),\\
T &\in \Hom(\g_{-1}\wedge\g_{-1},\g_{-1}).
\end{align*}
They satisfy the normality conditions:
\begin{align*}
\tr R - S - T &= 0;\\
\tr S + \tr T &= 0.
\end{align*}
\end{lem}

\begin{rem} These normality conditions determine both coefficients $E^i_{j[kl]}$ and
$F_{ij}$ uniquely assuming that $F$ has been normalized to be symmetric by the
appropriate choice of the section $s\colon M\to\G$. Indeed, according
to~\eqref{tensorS} and~\eqref{tensorT}, the condition $\tr S + \tr T = 0$ reduces
to a linear equation on the symmetric part of $F_{ij}$:
\[
\tr (\delta^i_{[k}F_{l]j}+F_{j[l}\delta^i_{k]}) = (1-l)(F_{jk}+F_{kj})=\dots
\]
where the right-hand side depends only on the known terms. This equation
determines the symmetric part of $F$ uniquely. Similarly, the condition $\tr R
- S - T = 0$ determines $E^i_{j[kl]}$ in a unique way.

Let us also comment on the geometric contents of the computed coefficients
$E$ and $F$. While the $E^i_{j[kl]}$ have completed  the
definition of the Weyl connection (describing the differentiation in the
$\g_{-2}$--directions as given by the splitting fixed by the coefficients
$C$, whereas $A^i_{jk}$ already fixed the differentials in the
$\g_{-1}$--directions), the symmetric coefficients $F_{ij}$ represent the
homogeneity two part of the Rho tensor.
\end{rem}

\section{Normality obstructions}
\label{sec:4}

Let us compute when our inclusion maps normal $B_l$-geometry to a normal
$D_{l+1}$-geometry. The map $\alpha\colon\g\to\tilde \g$ induces an
isomorphism:
\[
\bar\alpha \colon\g/\pp\to\tilde \g/\tilde\pp.
\]
On the other hand, we can identify $\g/\pp$ and $\tilde \g/\tilde\pp$ with
$\pp_+$ and $\tilde\pp_+$ respectively using Killing forms on $\g$ and
$\tilde\g$ respectively. This immediately implies that the induced isomorphism
(of vector spaces) $\phi\colon\pp_+\to \tilde\pp_+$ has the form:
\[
\phi\colon \pp_+\to \tilde\pp_+,\quad
\begin{pmatrix} 0 & 0 & 0 \\ - Z^t & 0 & 0 \\ T & Z & 0 \end{pmatrix}\mapsto
\begin{pmatrix} 0 & 0 & 0 & 0 \\ 0 & 0 & 0 & 0 \\
-\sqrt{2}Z^t & 0 & 0 & 0 \\ T & \sqrt{2}Z & 0 & 0
\end{pmatrix}.
\]

We can fix a similar basis $\tE^{[ij]}$, $\tE^i_j$, $\tE_{[ij]}$ in $\tilde\g$,
where the indices run now from $0$ to $l$, in such a way that
\begin{align*}
\phi(E^{[kl]}) &= \tE^{[kl]},\\
\phi(E^l) &= \sqrt{2} \tE^{[0l]},\\
\alpha(E_{[ij]} &= \tE_{[ij]},\\
\alpha(E_i) &= \frac{1}{\sqrt{2}}( \tE^0_i + \tE_{[0i]}),\\
\alpha(E^i_j) &= \tE^i_j.
\end{align*}

The problem of mapping normal $B_l$-connections to normal $D_{l+1}$-connections
can be reformulated as follows. Let $\kappa\in\wedge^2\pp_+\otimes\g$ be the
structure function of the normal $B_l$-connection., i.e., it satisfies the
equation $\partial^*\kappa=0$. The mapping $\phi$ can be naturally extended to
the morphism $\wedge^k\pp_+\otimes\g \to
\wedge^k\tilde\pp_+\otimes\tilde\g$, which is defined as the natural extension of
$\phi$ on $\wedge^k\pp_+$ and as the embedding $\alpha$ on $\g$. 

\emph{The normality question can be stated as follows: Under which
conditions is the element $\phi(\kappa)\in \wedge^2\tilde\pp_+\otimes
\tilde\g$ co-closed too?}

Let us introduce the operator $[\partial^*, \phi] = \partial^*\phi -
\phi\partial^*$. It acts from $\wedge^2\pp_+\otimes\g$ to
$\wedge^2\tilde\pp_+\otimes \tilde\g$. As $\kappa$ itself is coclosed, it is
clear that $\phi(\kappa)$ is coclosed if and only if $[\partial^*,
\phi](\kappa)=0$.

Define the following elements in $\tilde\g$:
\begin{align*}
\Delta\tE^i &= \tE^{[0i]} - \tE_0^i,\\
\Delta\tE_i &= \tE_{[0i]} - \tE^0_i.
\end{align*}

We shall also need the following commutation relations that are easy to compute:
\begin{align*}
[\Delta\tE^i, \alpha(E^r_s)] &= -\delta^i_s\Delta\tE^s;\\
[\Delta\tE^i, \alpha(E_s)] &= -\delta^i_s\Delta\tE^0_0;\\
[\Delta\tE^i, \alpha(E_{[rs]})] &= \delta^i_{[r}\Delta\tE_{s]};\\
[\Delta\tE^i, \alpha(E^r)] &= 0;\\
[\Delta\tE^i, \alpha(E^{[rs]})] &= 0.
\end{align*}

Let us compute the operator $[\partial^*, \phi]$ explicitly. For example, for
any $\kappa = E^{[ij]}\wedge E^{[kl]}\otimes X$, $X\in\g$ we have:
\begin{align*}
\phi(\kappa) &= \tE^{[ij]}\wedge \tE^{[kl]}\otimes \alpha(X);\\
\partial^*\phi(\kappa) &= \tE^{[ij]}\otimes [\tE^{[kl]}, \alpha(X)] - \tE^{[kl]}\otimes [\tE^{[ij]},
\alpha(X)];\\
\partial^*(\kappa) &= E^{[ij]}\otimes [E^{[kl]}, X] - E^{[kl]}\otimes
[E^{[ij]}, X];\\
\phi\partial^*(\kappa)&= \tE^{[ij]}\otimes \alpha([E^{[kl]},X]) -
\tE^{[kl]}\otimes \alpha([E^{[ij]},X]) = \\
&= \tE^{[ij]}\otimes [\tE^{[kl]}, \alpha(X)] - \tE^{[kl]}\otimes [\tE^{[ij]},
\alpha(X)].
\end{align*}
Here we use the fact that $\alpha$ is a homomorphism of Lie algebras. We get:
\begin{equation}\label{g22}
[\partial^*, \phi]\colon E^{[ij]}\wedge E^{[kl]}\otimes X\mapsto 0.
\end{equation}

Similarly, for $\kappa = E^i\wedge E^{[jk]}\otimes X$, $X\in\g$ we have:
\begin{align*}
\phi(\kappa) &= \sqrt{2}\tE^{[0i]}\wedge\tE^{[jk]}\otimes\alpha(X);\\
\partial^*\phi(\kappa) &= \sqrt{2}
\tE^{[0i]}\otimes[\tE^{[jk]},\alpha(X)]-\sqrt{2}
\tE^{[jk]}\otimes[\tE^{[0i]},\alpha(X)];\\
\partial^*(\kappa) &= E^i\otimes [E^{[jk]}, X] - E^{[jk]}\otimes
[E^i,\alpha(X)];\\
\phi\partial^*(\kappa) &= \sqrt{2}\tE^{[0i]}\otimes[\alpha(E^{[jk]}),
\alpha(X)] - \tE^{[jk]}\otimes [\alpha(E^i), \alpha(X)]\\
&= \sqrt{2} \tE^{[0i]}\otimes
[\tE^{[jk]},\alpha(X)]-\frac{1}{\sqrt{2}}\tE^{[jk]}\otimes [\tE^i_0,
\alpha(X)]\\
& -\frac{1}{\sqrt{2}}\tE^{[jk]}\otimes [\tE^{[0i]},\alpha(X)].
\end{align*}
Thus, we get:
\begin{equation}\label{g21}
[\partial^*,\phi]\colon E^i\wedge E^{[jk]}\otimes X \mapsto -\frac{1}{\sqrt{2}}
\tE^{[jk]}\otimes [\Delta\tE^i,\alpha(X)].
\end{equation}

Finally, for $\kappa = E^i\wedge E^j\otimes X$, $X\in \g$ we have:
\begin{align*}
\phi(\kappa) &= 2 \tE^{[0i]}\wedge \tE^{[0j]}\otimes\alpha(X);\\
\partial^*\phi(\kappa)&= 2\tE^{[0i]}\otimes [\tE^{[0j]},\alpha(X)] - 2\tE^{[0j]}\otimes
[\tE^{[0i]},\alpha(X)];\\
\partial^*\kappa &= E^i\otimes [E^j, X] - E^j\otimes [E^i, X] - E^{[ij]}\otimes
X;\\
\phi\partial^*(\kappa)&= \tE^{[0i]}\otimes [\tE^{[0j]}+\tE^j_0,\alpha(X)]\\
& -\tE^{[0j]}\otimes [\tE^{[0i]}+\tE^i_0,\alpha(X)] -
\tE^{[ij]}\otimes\alpha(X).
\end{align*}
Thus, we get:
\begin{multline}\label{g11}
[\partial^*,\phi]\colon E^i\wedge E^j\otimes X \mapsto \\
\tE^{[ij]}\otimes\alpha(X) + \tE^{[0i]}\otimes [\Delta\tE^j,
\alpha(X)]-\tE^{[0j]}\otimes [\Delta\tE^i, \alpha(X)].
\end{multline}

\begin{lem}\label{ker}
The kernel of the mapping
\[
[\partial^*,\phi]\colon \wedge^2\pp_+\otimes\g \to \wedge^2\tilde\pp_+\otimes
\tilde\g
\]
contains the following spaces:
\begin{enumerate}
\item[(a)] $\wedge^2 \g_2\otimes \g$;
\item[(b)] $E^i\wedge E^{[jk]} \otimes \h_i$, $i=1,\dots, l$, where $\h_i$ is a subalgebra in $\g$
is given by $\h_i = \{ X\in \g \mid [\Delta \tE^i, \alpha(X)] = 0\}$.
\end{enumerate}
In particular, $\cap_{i=1}^l \h_i = \g_1 + \g_2$, and the kernel of
$[\partial^*,\phi]$ contains $\g_1\otimes\g_2\otimes(\g_1 + \g_2)$.
Moreover, the intersection of the kernel with $\wedge^2\g_1\otimes \g$ is
trivial.
\end{lem}
\begin{proof}
Item~(a) follows immediately from~\eqref{g22}. To prove the 
rest we need to show
that the images of $\g_1\otimes\g_2\otimes\g$ and $\wedge^2\g_1\otimes\g$
under the map $[\partial^*,\phi]$ do not intersect.

Let $s\colon\tilde\g\to \tilde\g$ be a non-trivial second order automorphism
that stabilizes $\alpha(\g)$. Then it is easy to see that $\alpha(\g)$
coincides with the $+1$ eigenspace $\tilde\g^{1}(s)$ of $s$ in $\tilde \g$,
while the elements $\Delta\tE^i$ lie in $\tilde\g^{-1}(s)$. Hence, for any
element $X\in\g$ the bracket $[\Delta\tE^i,\alpha(X)]$ lies in
$\tilde\g^{-1}(s)$ and is either 0 or is linearly independent of $\alpha(Y)$
of any non-zero $Y\in\g$. According to equations~\eqref{g21}
and~\eqref{g11}, we see that any non-zero elements in the images of
$\g_1\otimes\g_2\otimes\g$ and $\wedge^2\g_1\otimes\g$ under the map
$[\partial^*,\phi]$ are indeed linearly independent.

In particular, the intersection of the kernel of $[\partial^*,\phi]$ with
$\wedge^2\g_1\otimes\g$ is trivial. And the intersection of this kernel with
$\g_1\otimes\g_2\otimes\g$ is described by item~(b).
\end{proof}

Let us decompose $\kappa$ as $\kappa_{1,1}+\kappa_{1,2}+\kappa_{2,2}$, where
$\kappa_{i,j}\in \g_i\wedge\g_j\otimes\g$. Note that this decomposition is
$G_0$-invariant, but not, in general, $P$-invariant.

Up to now we have not used the fact that $\kappa$ is coclosed and is
concentrated in the positive degree of the space $\Hom(\wedge^2(\g_{-}),\g)$.
Using these additional facts, we arrive at the following result.

\begin{thm}\label{t:kappa11}
The extension of $B_l$-geometry to $D_{l+1}$-geometry is normal if and only if
$\kappa_{1,1}$ vanishes identically.
\end{thm}
\begin{proof}
According to Lemma~\ref{ker} the condition $[\partial^*,\phi](\kappa)=0$
implies that $\kappa_{1,1}=0$. Let us prove the converse.

The assertion (a) of 
Lemma~\ref{ker} implies $[\partial^*,\phi](\kappa_{2,2})=0$, 
and it remains to prove that $[\partial^*,\phi](\kappa_{1,2})=0$.

Let us decompose $\kappa=\sum_{i>0}\kappa^i$ according to the homogeneity. In
particular, we have the decomposition $\kappa_{1,2}=\sum_{i>0}\kappa^i_{1,2}$.

The Bianchi identity can be written as:
\[
\partial\kappa(X,Y,Z) = \left\{ \kappa(\kappa(X,Y),Z) \right\}
- \left\{ \big(L_{Z^*}\big)\kappa(X,Y) \right\},
\]
where the bracket $\{,\}$ denotes the complete anti-symmetrization by $X,Y,Z$.
Applying this formula to the case $X,Y,Z\in\g_{-1}$, we immediately see that
the right-hand side vanishes identically due to the assumption
$\kappa_{1,1}=0$. Thus, $\partial\kappa$ vanishes identically on
$\wedge^3\g_{-1}$. In more detail, for $X,Y,Z\in\g_{-1}$ we have:
\begin{equation}\label{dk12}
\partial\kappa(X,Y,Z) = \{ \kappa([X,Y],Z) \} - \{ [X,\kappa(Y,Z)] \}
= \{ \kappa_{1,2}([X,Y],Z) \} = 0.
\end{equation}
Note that the map $X\wedge Y\mapsto [X,Y]$ establishes an isomorphism of
$\wedge^2\g_{-1}$ and $\g_{-2}$. Then equation~\eqref{dk12} means that the
complete anti-symmetrization of $\kappa_{1,2}$ interpreted as an element of
$\g_1\otimes\wedge^2\g_1\otimes\g$ is identically $0$.

Next, consider the normality condition $\partial^*\kappa=0$ in more detail.
Denote by $\pr_1$ the projection of $\pp_{+}\otimes\g$ to $\g_1\otimes\g$
along $\g_2\otimes\g$. Note that $\pr_1\partial^*(\kappa_{2,2})=0$. 
Since $\kappa_{1,1}=0$ by
assumption, the equation $\partial^*\kappa=0$ implies that
$\pr_1\partial^*(\kappa_{1,2})=0$. In more detail, $\pr_1\partial^*$ restricted
to $\g_{1}\otimes\g_{2}\otimes\g$ and applied to $\kappa$ can be written as:
\begin{equation}\label{prk12}
\pr_1\partial^*\colon \g_{1}\otimes\g_{2}\otimes\g\to \g_{1}\otimes\g,
\quad X\otimes Y\otimes Z \mapsto X\otimes[Y,Z],\end{equation}
where $
X\in\g_1,Y\in\g_2,Z\in\g$.

Now, let us prove $[\partial^*,\phi](\kappa^i_{1,2})=0$ for each $i$. As
$\kappa^i_{1,2}$ takes values in $\g_{i-3}$, according to Lemma~\ref{ker} this
equation is satisfied automatically for $i>3$. It remains to consider the cases
$i=1,2,3$.

For $i=1$ we have $\kappa^1_{1,2}\in \g_1\otimes\g_2\otimes \g_{-2}$, and it
is the only harmonic part of the curvature. As the space of harmonic curvatures
coincides with the traceless part of $\g_1\otimes\g_2\otimes \g_{-2}$, we
derive that all traces of $\kappa^1_{1,2}$ should vanish.

We can explicitly write $\kappa^1_{1,2}$ as $P^{[rs]}_{i[jk]}E^i\otimes
E^{[jk]}\otimes E_{[rs]}$. Then according to~\eqref{g21} we have:
\[
[\partial^*,\phi](\kappa^1_{1,2}) = -\frac{1}{\sqrt{2}} P^{[rs]}_{i[jk]}
\tE^{[jk]}\otimes [\Delta\tE^i,\alpha(E_{[rs]})] = \sqrt{2} \sum_i
P^{[is]}_{i[jk]} \tE^{[jk]}\otimes \tE_s.
\]
As the tensor $P$ is totally traceless (this can also be derived directly
from~\eqref{dk12} and~\eqref{prk12}), we see that $\sum_i P^{[is]}_{i[jk]}=0$
and, thus, $[\partial^*,\phi](\kappa^1_{1,2})=0$.

Next, for $i=2$ we have $\kappa^2_{1,2} = S^l_{i[jk]}E^i\otimes E^{[jk]}\otimes
E_l$. Then according to~\eqref{g21}:
\[
[\partial^*,\phi](\kappa^2_{1,2}) = -\frac{1}{\sqrt{2}} S^l_{i[jk]}
\tE^{[jk]}\otimes [\Delta\tE^i,\alpha(E_l)] = \frac{1}{\sqrt{2}} \sum_i
S^i_{i[jk]} \tE^{[jk]}\otimes \tE^0_0.
\]
But equation~\eqref{dk12} implies that $S^l_{[ijk]}=0$, and the condition
$\pr_1\partial^*(\kappa^2_{1,2})=0$ is rewritten as:
\[
S^l_{i[jk]}E^i\otimes [E^{[jk]}, E_l] = - \sum_j S^j_{i[jk]} E^i\otimes E^k =
0.
\]
Hence, $\sum_j S^j_{i[jk]}=0$ and together with $S^l_{[ijk]}=0$ it implies that
$\sum_i S^i_{i[jk]}=0$. Thus, $[\partial^*,\phi](\kappa^2_{1,2})=0$ as well.

Similarly, for $i=3$ we have $\kappa^3_{1,2} = Z^s_{i[jk]r}E^i\otimes
E^{[jk]}\otimes E^r_s$. Then according to~\eqref{g21}
\[
[\partial^*,\phi](\kappa^3_{1,2}) = -\frac{1}{\sqrt{2}} Z^s_{i[jk]r}
\tE^{[jk]}\otimes [\Delta\tE^i,\alpha(E^r_s)] = - \frac{1}{\sqrt{2}} \sum_i
Z^i_{i[jk]r} \tE^{[jk]}\otimes \tE^r.
\]
As above, equation~\eqref{dk12} implies that $Z^s_{[ijk]r}=0$, and the
condition $\pr_1\partial^*(\kappa^3_{1,2})=0$ is rewritten as:
\[
Z^s_{i[jk]r}E^i\otimes [E^{[jk]}, E^r_s] = - \sum_j Z^j_{i[jk]r} E^i\otimes
E^{[rk]} = 0.
\]
Hence, $\sum_j Z^j_{i[jk]r}=0$ and together with $Z^s_{[ijk]r}=0$ it implies
that $\sum_i Z^i_{i[jk]r}=0$. Thus, $[\partial^*,\phi](\kappa^2_{1,2})=0$.

This completes the proof of the theorem.
\end{proof}

Geometrically, the condition $\kappa_{1,1}=0$ from Theorem~\ref{t:kappa11}
means that the curvature tensor of the normal Cartan connection associated with
the non-degenerate distribution $D$ vanishes identically on $\wedge^2 D$. It is
clear that this condition is well-defined and defines a certain subclass in the
class of all non-degenerate $l$-dimensional distributions on
$l(l+1)/2$-dimensional manifolds. The example from S.~Armstrong~\cite{armst}
shows that this subclass is not empty.

\begin{ex}[\cite{armst}]\label{ex-a}
Let $l$ be any integer greater or equal to $4$. Let $\{x_i$, $y_{[jk]}\}$ be a
local coordinate system on $M$, where $1\le i,j,k\le l$, $j<k$. We also define
functions $y_{[kj]}$ as $-y_{[jk]}$ for $k>j$. Define a frame:
\[
Y_{[jk]} = \dd{y_{[jk]}},\quad X_i = \dd{x_i}-\sum_{p=i+1}^l x_p Y_{[ip]}.
\]
$1\le j<k\le l$, $1\le 1\le l$. As before, we define also $Y_{[kj]}=-Y_{[jk]}$
for $k>j$ and $Y_{[jj]}=0$ for any $1\le j\le l$.

Clearly, we have $[X_i,X_j]=Y_{[ij]}$ and $[X_i,Y_{[jk]}]=0$ for all $1\le
i,j,k\le l$ The distribution $D$ spanned by $X_1,\dots,X_l$ defines a flat
normal Cartan connection of type $B_l$ and has a maximal possible symmetry
algebra of dimension $l(l+1)/2$ among all non-degenerate distributions of
rank~$l$.

Now let $X_1'=X_1+y_{[12]}Y_{[34]}$ and $X_i'=X_i$ for $i\ge 2$. Define $D'$ as
a span of $X_1',X_2',\dots,X_l'$. Is is easy to see that we still have the
commutation relation $[X_i',X_j']=-Y_{[ij]}$ for all $1\le i,j\le l$. But now
we also get an additional non-trivial relation $[X_1',Y_{[12]}]=-Y_{[34]}$.

Thus, using structure equations~\eqref{streq}, we see that the only
non-vanishing structure coefficient is $f_{1[12]}^{[34]}=1$. Note that the
tensor $P$ defined by~\eqref{tensorP} with $A=0$ is already trace-free. So, we
see that functions $A^i_{jk}$ and $C^i_{[jk]}$ vanish in this case, and the
only non-vanishing coefficient of the tensor $P$ is $P_{1[12]}^{[34]}=1$.
Proceeding to the coefficients of degree~2, in the same way we get $E=F=0$ and
all curvature parts of degree~2 (tensors $R$, $S$ and $T$) vanish identically.
Thus, we get $\om^i_j=0$ and $\om_i=0\mod D^{\perp}$.

In fact, assuming that $\om_i=0$ and $\om_{[ij]}=0$ (on $M$), we get a
$\g$-valued 1-form on $M$, whose curvature is concentrated in degree~1. We can
always extend this 1-form in a unique way to a well-defined Cartan connection
$\om$ on the direct product $\G=M\times P$.

In particular, the structure function $\kappa$ of the constructed normal
connection satisfies the condition $\kappa_{1,1}=0$, and the connection itself
extends to a normal almost spinorial Cartan geometry. On the other hand,
$\kappa\ne0$, and the distribution~$D'$ is not equivalent to the model
distribution~$D$.
\end{ex}

\end{document}